\documentclass{amsart}
\usepackage{graphicx} 
\usepackage{amssymb}
\usepackage{float}

\numberwithin{equation}{section}

\newtheorem{theorem}{Theorem}[section]
\newtheorem{lemma}[theorem]{Lemma}
\newtheorem{remark}[theorem]{Remark}

\newcommand{\bx}{\boldsymbol{x}}

\title[PINN-enriched FEM]{A PINN-enriched Finite Element Method for Linear Elliptic Problems}
\author{Xiao Chen}
\address{School of Mathematical Science and Technology, University of Science and Technology of China, Hefei 230026, China.}
\email{xiao\_chen@mail.ustc.edu.cn}

\author{Yixin Luo}
\address{Suzhou Institute for Advanced Research, University of Science and Technology of China, Suzhou 215000, China.}
\email{seeing@ustc.edu.cn}

\author{Jingrun Chen}
\address{School of Mathematical Science and Technology, University of Science and Technology of China, Hefei 230026, China.}
\email{jingrunchen@ustc.edu.cn}

\date{March 2025}

\begin{document}

\begin{abstract}
    In this paper, we propose a hybrid method that combines finite element method (FEM) and physics-informed neural network (PINN) for solving linear elliptic problems. This method contains three steps: (1) train a PINN and obtain an approximate solution $u_{\theta}$; (2) enrich the finite element space with $u_{\theta}$; (3) obtain the final solution by FEM in the enriched space. In the second step, the enriched space is constructed by addition $v + u_{\theta}$ or multiplication $v \cdot u_{\theta}$, where $v$ belongs to the standard finite element space. We conduct the convergence analysis for the proposed method. Compared to the standard FEM, the same convergence order is obtained and higher accuracy can be achieved when solution derivatives are well approximated in PINN. Numerical examples from one dimension to three dimensions verify these theoretical results. For some examples, the accuracy of the proposed method can be reduced by a couple of orders of magnitude compared to the standard FEM.
\end{abstract}

\maketitle

\section{Introduction}

Linear elliptic partial differential equations (PDEs)~\cite{hanlin} demonstrate importances in various applications of science and engineering, including fluid flow, heat conduction, and electrostatics. They are characterized by their solutions' smoothness and the existence of a unique solution under appropriate boundary conditions. A plethora of methods have been proposed for numerically solving linear elliptic PDEs, and we focus our attention on the finite element method (FEM)\cite{citeulike:8936200,Ciarlet2002TheFE,johnson1988numericalso}, physics-informed neural networks (PINN)\cite{raissi2019686} as well as Deep Ritz method\cite{weinan2017thedr}. 

The finite element method is a numerical technique for solving PDEs that has been widely used in engineering and physics due to its flexibility and accuracy. It involves approximating the solution over small, simple geometric elements, which approximate the entire domain of interest. 

 The standard finite element method involves the construction of a set of basis functions that span the solution space, followed by the assembly of a system of equations that is then solved for the unknown coefficients. However, this approach can be computationally expensive and rely on the quality of mesh, especially for large-scale problems. Posteriori error estimate\cite{ainsworth19971} and adaptive mesh refinement\cite{tian2016242} has been widely used to improve the results and decrease the time cost. 

Recently, significant advancement has been achieved in numerical methods for PDEs with the introduction of physics-informed neural networks. These networks are a class of deep learning algorithms trained on the governing equations of the physical systems, rather than solely on the data.

PINNs have demonstrated the ability to approximate solutions to different types of PDEs\cite{jin2021109951}. They are particularly promising for problems where traditional numerical methods may struggle, such as those with complex geometries\cite{berg201828}, heterogeneous materials, or multi-scale features\cite{liu2020multiscaledn}. The key innovation of PINNs is the incorporation of physical laws directly into the learning process, which ensures that the learned solution satisfies the underlying equations exactly, without the need for a large amount of training data. 

However, neural network-based methods suffer from limited accuracy, and the training process is time-consuming. Due to the lack of an efficient optimization method for the training process, they are hard to train to arrive at the global optimal state, especially when the penalty term is used to handle the boundary condition. A number of new models \cite{lyu2022110930,sun2024115830,dong2021114129} have been proposed aiming to improve the performance of PINN.

Some previous works have explored methods based the balance of traditional methods and PINN. In~\cite{sun2024115830}, the author only optimizes the parameters in output layer and take neural network as basis in Discontinuous Galerkin method\cite{dgforell,mu2014432} to make training process a linear solver. In~\cite{pinn-fem20250114}, domain is decomposed into two non-overlapping regions: a finite element domain for boundary and a neural network domain for interior region. PINN is multiplied by finite element basis in finite element domain for enforcing the Dirichlet boundary conditions. In~\cite{franck2024113144}, PINN can be used as a prior to correct the basis for discontinuous Galerkin methods.

Along with this line, we propose to enrich the traditional finite element basis functions with PINNs. In this new method, there are three steps to numerically solve PDEs: (1) train a PINN and obtain an approximate solution $u_{\theta}$; (2) enrich the finite element space with $u_{\theta}$; (3) obtain the final solution by FEM in the enriched space. In the second step, the enriched space is constructed by addition $v + u_{\theta}$ or multiplication $v \cdot u_{\theta}$, where $v$ belongs to the standard finite element space. The theoretical and numerical results show that the enriched finite element method using PINNs offers better accuracy benefited by PINN and convergence order by FEM. 

The paper is organized as follows: In Section 2, we introduce the concept of Physics-Informed Neural Networks and classical finite element method; in Section 3, we present our method for enriching finite element basis functions with PINNs; in Section 4, we provide error analysis of the modified finite element method; in Section 5, we demonstrate the effectiveness of our approach through a series of numerical experiments, comparing the performance of the enriched finite element method against traditional FEM; finally, we conclude with a discussion of the potential applications of this method and the challenges that remain to be addressed in future research.

\section{Preliminaries}

We first introduce the physics-informed neural network, a novel numerical method based on neural networks for solving PDEs, and the classical finite element method. 

\subsection{Physics-informed neural networks}

Consider the linear elliptic equation 
\begin{equation}\label{2.1}
    \left\{
    \begin{aligned}
        -\nabla \cdot (a(\bx)\nabla u(\bx))+c(\bx)u(\bx) =& f(\bx), \quad & & \bx \in \Omega, 
        \\
        u(\bx)=&g(\bx), & & \bx \in \partial\Omega, 
    \end{aligned}
    \right.
\end{equation}
where $\Omega$ is smooth enough.

PINNs use the fact that classical fully-connected neural networks are smooth functions of their inputs, as long as their activation functions are also smooth, to approximate the solution to (2.1). The smoothness of neural networks makes sure that after adding it to finite element space, the space is still a subspace of Sobolev space. 

In contrast to traditional numerical methods such as the finite element method and finite difference method (FDM), where system parameters are determined by solving linear systems of equations, the degrees of freedom (DoF) in Physics-Informed Neural Networks correspond to the weights and biases of the neural network architecture. These parameters are obtained through the minimization of a predefined loss function. Unlike conventional methods, the optimization process in PINNs is inherently more complex, primarily due to the nonlinear nature of typical activation functions employed in neural networks. 

In our case, the PINN is a smooth neural network with one to three hidden layers that takes as the input the variable x. We denote the network $u_{\theta}$, which is parameterized by $\theta \in \Theta$.

The core idea of PINN is to transform the original PDE into the following minimization problem:
\begin{equation}\label{2.2}
    \theta = \mathop{\arg\min}\limits_{\theta \in \Theta} J(\theta),
\end{equation}
where $J(\theta) = J_{r}(\theta) + \lambda J_{b}(\theta)$.

We introduce two different terms in (2.2): $J_{r}(\theta)$represents the loss of residual of PDE and $J_{b}(\theta)$ represents the loss of boundary conditions. $\lambda$ is the penalty coefficient. The residual loss function is defined by
\begin{equation}\label{2.3}
    J_{r}(\theta)= \int_{\Omega }( -\nabla \cdot (a(x)\nabla u_{\theta}(x))+c(x)u_{\theta}(x)-f(x)) ^{2}\mathrm{d}x
\end{equation}
while the boundary loss function is given by
\begin{equation}\label{2.4}
    J_{b}(\theta)= \int_{\partial\Omega }( u_{\theta}(x)-g(x)) ^{2}\mathrm{d}x.
\end{equation}
There is another way to train a neural network, which is called the deep Ritz method. In this method, the loss function $J_{r}(\theta)$ is replaced by:
\begin{equation}\label{2.7}
    J_{R}(\theta)= \int_{\Omega }\frac{1}{2}a(x)|\nabla u_{\theta}(x)|^{2}+c(x)u_{\theta}^{2}(x)-u_{\theta}f(x)\mathrm{d}x.
\end{equation}

In practice, the integrals in (2.3) and (2.4) are approximated using a Monte-Carlo method. This method relies on sampling a certain number of so-called ``collocation points'' in order to approximate the integrals. For computational simplicity, we restrict our domain to rectangular (in 2D) or cuboidal (in 3D) geometries, which enables the implementation of uniform collocation points as a practical alternative. Then, the minimization problem is solved using a gradient-type method, such as ADAM or L-BFGS, which corresponds to the learning phase.

After introducing physics-informed neural networks, we give a brief introduction of classical finite element method. 

\subsection{classical finite element scheme}
We present the classical finite element scheme for discretizing the PDE(2.1) with $g=0$ for simplicity.

1. We first write the variational formulation of PDE(2.1): we say u $\in H_{0}^{1}(\Omega)$ is a weak solution of problem(2.1) provided
\begin{equation}\label{vf}
    \int_{\Omega} a(x)\nabla u(x) \cdot \nabla v(x)+c(x)u(x)v(x) \mathrm{d}x=\int_{\Omega}f(x)v(x) \mathrm{d}x
\end{equation}
for all $v \in V = H_{0}^{1}(\Omega)$.

We denote 
\begin{equation}\label{2.9}
    B(u,v)=\int_{\Omega} a(x)\nabla u(x) \cdot \nabla v(x)+c(x)u(x)v(x) \mathrm{d}x,F(v)=\int_{\Omega}f(x)v(x) \mathrm{d}x.
\end{equation}

2. Let $V_{h} \subset H_{0}^{1}(\Omega)$ be any (finite-dimensional) subspace, with basis function $\{\phi_{1}, \phi_{2}, ..., \phi_{n}\}$,the solution of finite element method denoted by $u_{h} \in V_{h}$ satisfy the following equations:
\begin{equation}\label{2.10}
     \int_{\Omega} a(x)\nabla u_{h}(x) \cdot \nabla \phi_{i}(x)+c(x)u_{h}(x)\phi_{i}(x) \mathrm{d}x=\int_{\Omega}f(x)\phi_{i}(x) \mathrm{d}x \quad  1\le i \le n.
\end{equation}
Let $\mathcal{T}^{h}$ denote a subdivision of Ω, where $h$ denote the size of mesh: the length of interval in 1d, edge length of triangle in 2d and tetrahedron in 3d. The Lagrange finite element space $V_{h}$ is taken as the piecewise polynomial space, which is defined by:
\begin{equation}
    V_{h}=\{v\in H^{1}(\Omega)|\ \forall T\in \mathcal{T}^{h}, v|_{T} \in P^{k}(T)\}.
\end{equation}
adding the Dirichlet boundary condition, space is changed to 
\begin{equation}
    V_{h}=\{v\in H_{0}^{1}(\Omega)|\ \forall T\in \mathcal{T}^{h}, v|_{T} \in P^{k}(T)\},
\end{equation}
where $P^{k}$ denote the space of polynomials of degree less than k. 

3. Because $u_{h} \in V_{h}$,we can write $u_{h} = \sum_{j=1}^{n} u_{j}\phi_{j}$. Let $A_{ij}=B[\phi_{i},\phi_{j}]$, $F_{i}=F[\phi_{i}]$ for $i,j=1,...,n$. Set $\mathbf{U}=(u_{j}),\mathbf{A}=(A_{ij}),\mathbf{F}=(F_{i})$. Then equations(2.8) is equivalent to solving the (square) matrix equation
\begin{equation}\label{2.11}
    \mathbf{A} \mathbf{U} =\mathbf{F}. 
\end{equation}

\section{Enriching Approximation Space of FEM with PINN}
Having reviewed the classical finite element method, we recognize that enhancing the method's accuracy requires $V_{h}$ to effectively approximate the solution $u$. Conventionally, $V_{h}$ is constructed as a piecewise polynomial space with locally supported basis functions. In contrast, PINNs offer a fundamentally different approach, as they operate as global approximators. Motivated by this distinction, we propose a novel function space that incorporates a neural network component $u_{\theta}$, aiming to combine the strengths of both local and global approximation capabilities.

We have two ways to get a function space containing $u_{\theta}$. One is simply adding $u_{\theta}$ to space $V_{h}$.

\subsection{Additive Space}
We define the additive space:

\begin{equation}
    V_{h}^{+}=V_{h}+u_{\theta}=\{v|\exists w \in V_{h},v=w+u_{\theta}\}.
\end{equation}
We use $V_{h}^{+}$ for trial function space and still use $V_{h}$ for test function space. Then the finite element problem become:
\begin{equation}
\begin{cases}
    \text{Find } u_{h} \in V_{h}^{+}, \text{which leads to } w_{h}=u_{h}-u_{\theta} \in V_{h},
    \\
    \text{such that, }\forall v \in V_{h},B[w_{h},v]=F(v)-B[u_{\theta},v].
\end{cases}
\end{equation}
This is equivalent to solving the same equation with an exact solution replaced by $w=u-u_{\theta}$. So we could easily use the frameworks of classical FEM for analysis.

\subsection{Multiplicative Space}
 Here we focus another more complicated way: each basis function of $V_{h}$ is multiplied by the function $u_{\theta}$. We denote the multiplicative space by $V_{h}^{*}$, which satisfies:
\begin{equation}\label{3.1}
    V_{h}^{*} = V_{h}*u_{\theta} = \{v|\exists w \in V_{h},v=wu_{\theta} \} = \{v|v/u_{\theta} \in V_{h} \}(\text{when}\  u_{\theta}\ne 0).
\end{equation}
We hope that compared with $V_{h}$, $V_{h}^{*}$ has a better approximation of $u$. Before analyzing the approximation ability, we need to talk about the cases when $u_{\theta}$ has zero points.

\begin{lemma}
    If $x$ is a zero point of $u_{\theta}$, any function $v \in V_{h}^{*}$ has the same zero point $x$.
\end{lemma}
\begin{proof}
    For $\forall v \in \in V_{h}^{*}$, from (3.3), $\exists w \in V_{h}$ such that
    \begin{equation}\label{3.4}
        v=wu_{\theta}
    \end{equation}
    So $v(x)=w(x)v_{\theta}(x)=0$.
\end{proof}
\newtheorem{proposition}[theorem]{Proposition}
\begin{proposition}
    If $x$ is a zero point of $u_{\theta}$, we have
    \begin{equation}\label{3.5}
        \inf_{v\in V_{h}^{*}} ||u-v||_{L^{\infty}}\ge |u(x)|
    \end{equation}
\end{proposition}
\begin{proof}
    Lemma 3.1 shows that $\forall v \in V_{h}^{*}$, $v(x)=0$, $||u-v||_{L^{\infty}}\ge |u(x)-v(x)|=|u(x)|$. So  $\inf\limits_{v\in V_{h}^{*}} ||u-v||_{L^{\infty}}\ge |u(x)|$.
\end{proof}
Proposition 3.2 shows that the approximation error using $L_{\infty}$ norm has a lower bound when $u_{\theta}$ has a zero point. To solve this problem, we make some corrections to the space $V_{h}^{*}$. An easy way to avoid the influence of zero point is adding a constant to $u_{\theta}$ to remove all of the zero point.

For all the bounded $u_{\theta}$, there exists a constant $C$ such that $u_{\theta} + C$ has no zero point. For example, if $u_{\theta} \in [-a,b]$, $a$ and $b$ is positive, we can take $C$ as $\frac{a+b}{2}$. Then, we use the space with correction to approximate $u_{\theta}+C$, which means when we solve equation(2.1), we need to solve following equation firstly:
\begin{equation}\label{3.6}
    \begin{cases}
        -\nabla \cdot (a(x)\nabla u^{+}(x))+c(x)u^{+}(x) = f(x)+Cc(x) \quad \text{in} \ \Omega,
        \\
        u^{+}(x)=g(x)+C \quad \text{on} \ \partial\Omega.
    \end{cases}
\end{equation}
Then $u(x)$ is given by $u(x)=u^{+}(x)-C$.

Equation(3.6) is still a linear elliptic equation and $u^{+}_{\theta}=u_{\theta}+C$ has no zero point. So we can use space multiplied by $u^{+}_{\theta}$ as finite element space to solve equation(3.6) and then minus constant $C$ to obtain the numerical solution of equation(2.1).

\begin{remark}
    When $g(x)=0$ and the exact solution has no zero points in the domain $\Omega$, it remains necessary to introduce a constant term in our formulation. This precautionary measure addresses the potential existence of zero points in the neural network approximation $u_{\theta}$, particularly near the domain boundaries. We have made some numerical experiments, which have consistently demonstrated that such zero points can indeed emerge in boundary-adjacent regions, thereby justifying the inclusion of this additional constant term.
\end{remark}

\begin{remark}
An alternative approach to circumvent the issue of zero points involves local modification of the basis functions. Following mesh generation, we evaluate $u_{\theta}$ at each degree of freedom. When the computed value falls below a specified threshold, we retain the original basis function at that particular DoF rather than incorporating the product with $u_{\theta}$. Nevertheless, numerical experiments indicate that this localized adaptation strategy yields less substantial improvements compared to the previously discussed method.   
\end{remark}

Since we have used correction to avoid the influence of zero point, in the following discussion, we assume that $u_{\theta}$ is smooth and does not have a zero point, which indicates that $u_{\theta}^{-1}$ is smooth. In this case, the properties of $V_{h}$ are inherited by $V_{h}^{*}$. 

\begin{lemma}
    Assume that $u_{\theta}$ is smooth and has no zero point, $V_{h} \subset W_{p}^{k}(\Omega)$, $1\le p < \infty$. Then $V_{h}^{*}  \subset W_{p}^{k}(\Omega)$ and 
    \begin{equation}\label{3.9}
        ||vu_{\theta}||_{W_{p}^{k}(\Omega)} \le C_{n,k,p}||u_{\theta}||_{C_{k}(\Omega)}||v||_{W_{p}^{k}(\Omega)}, \quad \forall v \in W_{p}^{k}(\Omega).
    \end{equation}
    If we substitute $u_{\theta}^{-1}$ for $u_{\theta}$
    \begin{equation}\label{3.10}
        ||v/u_{\theta}||_{W_{p}^{k}(\Omega)} \le C_{n,k,p}||u_{\theta}^{-1}||_{C_{k}(\Omega)}||v||_{W_{p}^{k}(\Omega)}, \quad \forall v \in W_{p}^{k}(\Omega).
    \end{equation}
\end{lemma}
\begin{proof}
    using Leibniz's rule, the high order derivative is calculated by
    \begin{equation}\label{3.11}
        (fg)^{(\beta)}=\sum_{\alpha\le \beta}\prod_{i\le n}\binom{\beta_{i}}{\alpha_{i} }f^{(\alpha )}  g^{(k-\alpha )}  
    \end{equation}
    where n is dimension, $\beta=(\beta_{1},\beta_{2},...,\beta_{n})$, $\alpha= (\alpha_{1},\alpha_{2},...,\alpha_{n})$.
    \begin{equation}\label{3.12}
        \begin{split}
            ||vu_{\theta}||_{W_{p}^{k}(\Omega)}^{p}&=\sum_{|\alpha| \le k}\int_{\Omega}(|(vu_{\theta})^{(\alpha)}|^{p}\mathrm{d} x \\
            &\le C_{n,k,p}\sum_{|\alpha+\beta|\le k }\int_{\Omega} |v^{(\alpha )}u_{\theta}^{\beta)}|^{p}\mathrm{d}x\\
            &\le C_{n,k,p}||u_{\theta}||_{C_{k}(\Omega)}^{p}\sum_{|\alpha|\le k }\int_{\Omega}|v^{(\alpha )}|^{p}\mathrm{d}x\\
            &=C_{n,k,p}||u_{\theta}||_{C_{k}(\Omega)}^{p}||v||_{W_{p}^{k}(\Omega)}^{p}.
        \end{split}
    \end{equation}
\end{proof}
Lemma 3.5 establishes that the modified finite element remains conforming. This crucial property enables us to employ analogous analytical techniques for error estimates in section 4.

\subsection{Dealing with Boundary Conditions}
As the function space changes, so do the corresponding boundary conditions. Here, we only talk about the Dirichlet boundary conditions because Neumann and robin conditions can be converted directly into part of the variational formulation. 

For additive space, as mentioned in (3.2), the boundary condition is replaced by:
\begin{equation}
    w(x)=g(x)-u_{\theta}(x), \quad x \in \partial \Omega. 
\end{equation}
When PINN exactly satisfies the boundary condition, we can easily know $w(x)=0$ on boundary.

For multiplicative space, the way to handle boundary condition is similar. We define $w(x)=u(x)/u_{\theta}(x)$. The boundary conditions of $w(x)$ become:
\begin{equation}
    w(x)=g(x)/u_{\theta}(x), \quad x \in \partial \Omega. 
\end{equation}
When PINN exactly satisfies the boundary condition, $w(x)$ satisfies $w(x)=1$ on boundary.

\begin{remark}
    There is a special case when $g(x)=0$ and $u_{\theta}(x)=0$ on boundary. If we do not add a constant to the solution, any function in space $V_{h}^{*}$ satisfies zero boundary condition. So we can ignore the homogeneous condition. However, as we have mentioned in remark 3.3, we can't make sure that $u_{\theta}$ has no zero points. So (3.12) is used after adding a constant.
\end{remark}

\section{Error Analysis}
In this section, we prove that the modified FEM provides the same order as classical FEM and improve its accuracy depending on the error of PINN. We first analyze the error of PINN. 

\subsection{Error Analysis of PINN}

The $L_{2}$ error of PINN can be controlled by three parts: optimization error, generalization
error, and approximation error. Here, we mainly focus on the optimization error and approximation error. For simplicity, we set $\Omega$ to be $[0,1]^{d}$ for normalization. Following the notation of Kutyniok\cite{kuty}, 
to represents the ability of NN to approximate the exact value, we use the $L_{2}$ error of ${u}_{\theta }$, which is defined as:
\begin{equation}\label{4.1}
    \mathcal{R}(u_{\theta }) : = \int_{\Omega }( {u}_{\theta }(z_{i})-u(z_{i}) )^{2}\mathrm{d}x.
\end{equation}
In the training process, we use Adam to minimize the non-convex loss function. In most cases, we can not find the global minimum. Because the objective function is non-convex, the optimization error is unknown. Unfortunately, it is usually the largest one. There are some ways to decrease optimization error, by using random neural network, which make the loss function become quadratic form, or using other optimizer such as L-BFGS. We propose a training strategy to decrease the optimization error in Appendix A.


What is left is the approximation error. It is defined as:
\begin{equation}
   \varepsilon_{a}:=  \sup_{\theta \in \Theta } \mathcal{R}(u_{\theta })
\end{equation}
Approximation error shows the ability of neural network to approximate the exact solution. It is well studied and the results are great. Ingo Gühring et al\cite{guhring2021107} has provided the approximation rates for neural network in Sobolev spaces. In order not to lose integrity, we give his main results.
\begin{theorem}
Let d be dimension of function, $j,\tau \in \mathbb{N}_{0},k\in {0,...,j},n\in \mathbb{N}_{\ge k+1},1\le p\le \infty $ and $\mu > 0$, activation function is $\sin$ or $\tanh$. Then, there exist constants $L,C$ and $\tilde{\epsilon} $ depending on $d,n,k,p, \mu$, we have following properties.

For every $\epsilon \in(0,\tilde{\epsilon})$ and every $f \in W^{n,p}((0,1)^{d})$,there is a neural network $\phi_{\epsilon ,f}$with d-dimensional input,one-dimensional output, at most L layers and at most $M_{\epsilon}=C\epsilon^{-d/(n-k-\mu_{(k=2)})}$nonzero weights, such that
\begin{equation}
    \parallel\phi_{\epsilon ,f} -f \parallel _{W^{k,p}((0,1)^{d})}\le \epsilon \parallel f\parallel _{W^{n,p}.((0,1)^{d})}
\end{equation}
\end{theorem}

Because $\mu$ is arbitrary, Theorem 4.1 actually shows that order of convergence is $\frac{n-k}{d}$. This is the same as the order of FEM using high order finite element space.

\begin{remark}
    When f is analytic, we can not simply say that when n tends to infinity, $M_{\epsilon}$ tends to be a constant, since neural network can not approximate every analytic function exactly. De Ryck et al\cite{deryck2021732} shows that tanh neural networks with two hidden layers result in an exponential convergence rate, numerical results can be seen in \cite{sun2024115830}.
\end{remark}

\subsection{Error Analysis of Modified FEM}
Here, we mainly give the error of FEM using multiplicative space. Error analysis of FEM using additive space has no difference with that of classical FEM so we give its results directly. Let $u$ be the solution to the variational problem (2.6) and $u_{h}$ be the solution to the approximation problem (2.8). We now want to estimate the error $||u-u_{h}||_{V}$. We do so by the following theorems. 

\begin{theorem}[\textbf{Lax-Milgram}]
    Given a Hilbert space$(V,(.,.))$, a continuous, coercive bilinear form $B(u,v)$ and a continuous linear functional $F \in V^{'}$ satisfying
    \begin{equation}
        |B(u,v)| \le  \alpha ||u||_{V} \ ||v||_{V} \quad (u,v \in V)
    \end{equation}
    and
    \begin{equation}
        \beta ||u||_{V}^{2} \le |B(u,u)| \quad u \in V.
    \end{equation}
    Then, there exists a unique $u \in V$ such that 
    \begin{equation}
        B(u,v) = F(v) \quad \forall v \in V.
    \end{equation}
\end{theorem}

\begin{theorem}[$\mathbf{C\acute{e} a}$]
    Suppose the conditions in theorem 3.4 hold and that u solves (3.1). For the finite element variational problem (3.3) we have 
    \begin{equation}
        ||u-u_{h}||_{V} \le \frac{\alpha}{\beta} \min_{v\in V_{h}}||u-v||_{V}.
    \end{equation}
\end{theorem}

By applying the preceding theorems, we establish the primary theoretical result: the finite element method employing spaces $V_{h}$ , $ V_{h}^{+}$ and $V_{h}^{*}$ achieves identical convergence orders. 

let $\mathcal{I}_{h}$ denote a global interpolator for a family of finite elements in space $V_{h}$ based on the components of $\mathcal{T}^{h}$.
\begin{theorem}
    Suppose that conditions in theorem 4.3 hold and that the corresponding shape functions have an approximation order, m, that is
    \begin{equation}
        ||u-\mathcal{I}_{h}u||_{H^{1}(\Omega)} \le Ch^{m-1}|u|_{H^{m}(\Omega)}.
    \end{equation}
    Then, the unique solution $u_{h}^{*} \in V_{h}^{*}$ to the variational problem 
    \begin{equation}
        B[u_{h}^{*},v]=F(v) \quad \forall v \in V_{h}^{*}
    \end{equation}
    satisfies
    \begin{equation}
        ||u-u_{h}^{*}||_{H^{1}(\Omega)} \le Ch^{m-1}||u_{\theta}||_{C_{1}}|\frac{u}{u_{\theta}}|_{H^{m}(\Omega)}.
    \end{equation}
\end{theorem}

\begin{proof}
    From theorem 4.3, we obtain
    \begin{equation*}
        \begin{split}
            ||u-u_{h}^{*}||_{H^{1}(\Omega)} &\le C\min_{v \in V_{h}^{*}}||u-v||_{H^{1}(\Omega)} 
            \\
            (\mathcal{I}_{h}^{*}u \in V_{h}^{*})& \le ||u-\mathcal{I}_{h}^{*}u||_{H^{1}(\Omega)} 
            \\
            & = ||u_{\theta}(\frac{u}{u_{\theta}}-\mathcal{I}_{h}\frac{u}{u_{\theta}})||_{H^{1}(\Omega)} 
            \\
            (lemma\ 3.5)& \le C||u_{\theta}||_{C_{1}}||\frac{u}{u_{\theta}}-\mathcal{I}_{h}\frac{u}{u_{\theta}}||_{H^{1}(\Omega)} 
            \\
            (4.8) & \le  Ch^{m-1}||u_{\theta}||_{C_{1}}|\frac{u}{u_{\theta}}|_{H^{m}(\Omega)}. 
        \end{split}
    \end{equation*}
\end{proof}

\begin{theorem}
    For additive space, with the same conditions in theorem 4.5, the solution $u_{h} \in V_{h}^{+}$ to problem (3.2) satisfies 
    \begin{equation}
        ||u-u_{h}||_{H^{1}(\Omega)} \le Ch^{m-1}|u-u_{\theta}|_{H^{m}(\Omega)}.
    \end{equation}
\end{theorem}

\begin{remark}
    Theorem 3.6 shows that the error is controlled by term $|\frac{u}{u_{\theta}}|_{H^{m}(\Omega)}$. Because m is positive, the derivative remains the same when it minus 1. we have $|\frac{u}{u_{\theta}}|_{H^{m}(\Omega)}=|\frac{u-u_{\theta}}{u_{\theta}}|_{H^{m}(\Omega)}\le C||u_{\theta}^{-1}||_{C_{m}(\Omega)}||u-u_{\theta}||_{H^{m}(\Omega)}$. This indicates that $H^{1}$error is controlled by $||u-u_{\theta}||_{H^{m}(\Omega)}$, which has been analyzed above. Therefore, the fundamental distinction between the classical and modified function spaces lies solely in their respective coefficients, while both PINN-enriched spaces yield comparable numerical results.
\end{remark}

We now consider error estimates for $u-u_{h}^{*}$ in the L2 norm. To estimate $||u-u_{h}^{*}||_{L^{2}(\Omega)}$, we need to use a “duality” argument. Let $\omega$ be the weak solution of following equation:
\begin{equation}
    \begin{cases}
        -\nabla \cdot (a(x)\nabla w(x))+c(x)w(x) = u-u_{h}^{*} \quad \text{in} \ \Omega,
        \\
        w(x)=0 \quad \text{on} \ \partial\Omega.
    \end{cases}
\end{equation}
The variational formulation of this problem is: find $\omega \in V$ such that
\begin{equation}
    B(w,v)=(u-u_{h}^{*},v)_{L_{2}}.
\end{equation}
By theorem 3.4 and 3.5 the solution is unique and by the regularity of solution\cite{evans} $w$ satisfies:
\begin{equation}
    |w|_{H^{2}(\Omega)} \le C||u-u_{h}^{*}||_{L^{2}(\Omega)}.
\end{equation}
With these properties, we can prove the following:

\begin{theorem}
    Suppose the conditions in theorem 4.5 hold. Then, the unique solution $u_{h}^{*}$ satisfies
    \begin{equation}
        ||u-u_{h}^{*}||_{L^{2}(\Omega)} \le Ch^{m}||u_{\theta}||_{C_{1}}|\frac{u}{u_{\theta}}|_{H^{m}(\Omega)}
    \end{equation}
    and similarly, solution $u_{h} \in V_{h}^{+}$ satisfies
    \begin{equation}
        ||u-u_{h}||_{L^{2}(\Omega)} \le Ch^{m}|u-u_{\theta}|_{H^{m}(\Omega)}.
    \end{equation}
\end{theorem}

\begin{proof}
    \begin{equation}\label{3.24}
    \begin{split}
        ||u-u_{h}^{*}||_{L^{2}(\Omega)}^{2} & = (u-u_{h}^{*}, u-u_{h}^{*}) 
        \\
        & = B[w,u-u_{h}^{*}]
        \\
        (B[u,v]=B[u_{h}^{*},v],v\in V_{h}^{*})& = B[w-\mathcal{I}_{h}w, u-u_{h}^{*}]
        \\
        & \le C||u-u_{h}^{*}||_{H^{1}(\Omega)}||w-\mathcal{I}_{h}w||_{H^{1}(\Omega)}
        \\
        & \le Ch||u-u_{h}^{*}||_{H^{1}(\Omega)}|w|_{H^{2}(\Omega)}
        \\
        & \le Ch||u-u_{h}^{*}||_{H^{1}(\Omega)}||u-u_{h}^{*}||_{L^{2}(\Omega)}.
    \end{split}
    \end{equation}
    Divide by $||u-u_{h}^{*}||_{L^{2}(\Omega)}$ on both sides, we have
    \begin{equation}\label{3.25}
        ||u-u_{h}^{*}||_{L^{2}(\Omega)} \le  Ch||u-u_{h}^{*}||_{H^{1}(\Omega)} \le Ch^{m}||u_{\theta}||_{C_{1}}|\frac{u}{u_{\theta}}|_{H^{m}(\Omega)}.
    \end{equation}
\end{proof}

The preceding theorems demonstrate that the approximation error is bounded by $|\frac{u}{u_{\theta}}| _{H^{m}(\Omega)}$ or $||u-u_{\theta}||_{H^{m}(\Omega)}$. In the ideal case where $u_{\theta}=u$, vanishes completely since 
$\frac{u}{u_{\theta}}$ becomes unity, and consequently, its derivatives evaluate to zero. This error estimate further reveals an inverse relationship between the PINN approximation error and the solution accuracy: as the PINN error decreases, the precision of the solution improves. 

However, when employing higher-order finite elements, the overall error becomes dependent on the PINN approximation error measured in the sense of $H^{m}$ seminorm. It is well-established in approximation theory that the convergence rate deteriorates when considering higher-order derivatives or operating in higher dimension. Consequently, the potential improvement offered by this approach may become less pronounced when using high-order elements, particularly in multidimensional settings.

\section{Numerical Results}
In this section, we present several test problems to demonstrate the performance of the modified finite element method and compare the numerical results with our theoretical analysis. All programs including PINN training and FEM framework run on a PC with Core i5-13400F CPU and Nvidia RTX 4060 Ti 8G GPU.

For one-dimensional problems, we employ a shallow network architecture consisting of a single hidden layer containing 20 neurons. In contrast, for problems in two or higher dimensions, we utilize a deeper architecture comprising three hidden layers with 20, 40, and 20 neurons, respectively. Additional architectural and training specifications are comprehensively documented in Table 1.

The numerical implementation employs Gaussian quadrature for computing integrals in the variational formulations, while mesh generation is handled by FEALPy. To mitigate potential issues arising from zero points in the solution domain, test solutions we choose are positive throughout the computational domain.

\begin{table}[ht]
    \centering
    \begin{tabular}{|c|c|c|c|} \hline 
         dimensions&  1d&  2d&  3d\\ \hline 
         layers-MLP&  \{1,20,1\} &  \{2,20,40,20,1\} &  \{3,20,40,20,1\}   \\ \hline 
         parameters&  61 &  1741 &  1761   \\ \hline 
         activate function& tanh & tanh & tanh  \\ \hline 
         lr& 2e-03&1e-03&3e-03 \\ \hline 
         loss function& $J_{r}(\theta)$& $J_{r}(\theta)$ or $J_{R}(\theta)$&$J_{R}(\theta)$ and $J_{r}(\theta)$ \\ \hline 
         epochs& 10000 & 10000 & 15000+10000 \\ \hline 
         collection points& 1000 & 400 &1000 \\ \hline 
    \end{tabular}
    \caption{Network structure and training parameters for different dimensions.}
    \label{tab:0}
\end{table}

Given the comparable performance of spaces $V_{h}^{+}$ and $V_{h}^{*}$, we present numerical results for both spaces in one-dimensional cases. For higher-dimensional problems, we restrict our presentation to results obtained using $V_{h}^{*}$. Due to computational complexity and memory constraints, we limit our numerical experiments to polynomial degrees not exceeding three in one-dimensional cases, two in two-dimensional cases, and one in three-dimensional cases.

The analysis presented in Section 4 reveals that the overall accuracy of our scheme is fundamentally determined by the capability of the PINN to approximate the derivatives of the solution. In our numerical implementation, we employ two complementary strategies to enhance the PINN's performance: (1) the design of a specialized boundary operator that exactly enforces boundary conditions, and (2) the implementation of a dual-loss-function approach that simultaneously accelerates training convergence and reduces approximation error. Comprehensive details regarding these implementations, along with comparative analyses, are provided in Appendix A.

\subsection{Poisson Equation in One Dimension}
The first problem we consider is the Poisson equation with Dirichlet boundary condition on \(\Omega = [0, 1]\) and the exact solution \( u = (1-x)\sin(5x)+2 \). We initiate our investigation by training a PINN with a learning rate of 0.003, comprising a total of 61 trainable parameters. Subsequently, we implement our modified finite element method using space $V_{h}^{+}$ and $V_{h}^{*}$, employing a uniform mesh discretization with $n$ intervals.

Table \ref{tab:1.1} to Table \ref{tab:1.6} present the computed $L^{2}$ and $H^{1}$ errors, along with their corresponding convergence rates,  for polynomial degrees $P^{1}$, $P^{2}$ and $P^{3}$. The comparative error analysis across different elements and norms is visualized in figure \ref{fig:log1}. Both $V_{h}^{+}$ and $V_{h}^{*}$ spaces demonstrate similar performance characteristics, consistently outperforming the classical finite element method in all considered cases.

\begin{figure}[H]
    \centering
    \begin{minipage}[t]{0.48\textwidth}
        \centering
        \includegraphics[width=\linewidth]{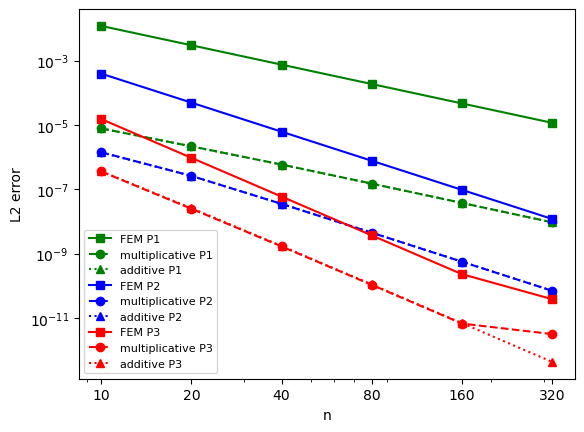}
    \end{minipage}
    \hfill 
    \begin{minipage}[t]{0.48\textwidth}
        \centering
        \includegraphics[width=\linewidth]{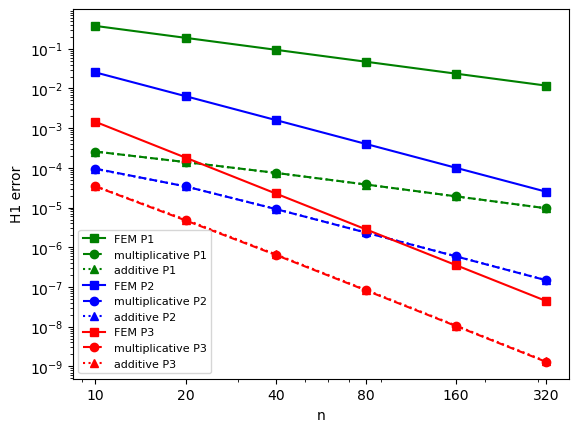}
    \end{minipage}
    \caption{left: $L^{2} error$ using different spaces, right: $H^{1} error$ using different spaces.}
    \label{fig:log1}
\end{figure}

\begin{table}[ht]
    \centering
    \setlength{\tabcolsep}{8pt}
    \begin{tabular}{|c|cc|cc|cc|} \hline 
         &  \multicolumn{2}{|c|}{$V_{h}$}&  \multicolumn{2}{|c|}{$V_{h}^{+}$}&  \multicolumn{2}{|c|}{$V_{h}^{*}$}\\ \hline 
         n&  $L^{2}$error &  order &  $L^{2}$error &  order &  $L^{2}$error &  order \\ \hline 
         10& 1.189e-02& -
&7.823e-06& - &7.811e-06& -
\\ \hline 
         20& 2.983e-03&1.995
&2.157e-06&1.859&2.175e-06&1.844
\\ \hline 
         40& 7.463e-04&1.999
&5.866e-07&1.878&5.912e-07&1.880
\\ \hline 
         80& 1.866e-04&2.000
&1.501e-07&1.966&1.513e-07&1.967
\\ \hline 
         160& 4.666e-05&2.000
&3.776e-08&1.991&3.804e-08&1.991
\\ \hline
         320& 1.166e-05&2.000&9.455e-09&1.998&9.523e-09&1.998\\ \hline
    \end{tabular}
    \caption{$L^{2}$error and order using $P^{1}$ Lagrange element, Poisson equation in 1d.}
    \label{tab:1.1}
\end{table}

\begin{table}[ht]
    \centering
    \setlength{\tabcolsep}{8pt}
    \begin{tabular}{|c|cc|cc|cc|} \hline 
         &  \multicolumn{2}{|c|}{$V_{h}$}&  \multicolumn{2}{|c|}{$V_{h}^{+}$}&  \multicolumn{2}{|c|}{$V_{h}^{*}$}\\ \hline 
         n&  $L^{2}$error &  order &  $L^{2}$error &  order &  $L^{2}$error &  order \\ \hline 
         10& 3.764e-01& -
&2.582e-04& - &2.578e-04& - 
\\ \hline 
         20& 1.887e-01&0.996
&1.394e-04&0.889&1.400e-04&0.880 
\\ \hline 
         40& 9.441e-02&0.999
&7.461e-05&0.902&7.492e-05&0.903 
\\ \hline 
         80& 4.721e-02&1.000
&3.804e-05&0.972&3.818e-05&0.973 
\\ \hline 
         160& 2.361e-02&1.000
&1.911e-05&

0.993&1.918e-05&0.993 
\\ \hline
         320& 1.180e-02&1.000&9.569e-06&0.998&9.603e-06&0.998 \\ \hline
    \end{tabular}
    \caption{$H^{1}$error and order using $P^{1}$ Lagrange element, Poisson equation in 1d.}
    \label{tab:1.2}
\end{table}

\begin{table}[ht]
    \centering
    \setlength{\tabcolsep}{8pt}
    \begin{tabular}{|c|cc|cc|cc|} \hline 
         &  \multicolumn{2}{|c|}{$V_{h}$}&  \multicolumn{2}{|c|}{$V_{h}^{+}$}&  \multicolumn{2}{|c|}{$V_{h}^{*}$}\\ \hline 
         n&  $L^{2}$error &  order &  $L^{2}$error &  order &  $L^{2}$error &  order \\ \hline 
         10& 3.924e-04& -  
&1.398e-06& -  
&1.429e-06& -  
\\ \hline 
         20& 4.940e-05&2.990
&2.625e-07&2.413&2.624e-07&2.445
\\ \hline 
         40& 6.186e-06&2.997
&3.565e-08&2.880&3.542e-08&2.889
\\ \hline 
         80& 7.736e-07&2.999
&4.556e-09&2.968&4.521e-09&2.970
\\ \hline 
         160& 9.671e-08&3.000
&5.727e-10&2.992
&5.682e-10&2.992
\\ \hline
         320& 1.209e-08&3.000&7.168e-11&2.998&7.113e-11&2.998\\ \hline
    \end{tabular}
    \caption{$L^{2}$error and order using $P^{2}$ Lagrange element, Poisson equation in 1d.}
    \label{tab:1.3}
\end{table}

\begin{table}[H]
    \centering
    \setlength{\tabcolsep}{8pt}
    \begin{tabular}{|c|cc|cc|cc|} \hline 
         &  \multicolumn{2}{|c|}{$V_{h}$}&  \multicolumn{2}{|c|}{$V_{h}^{+}$}&  \multicolumn{2}{|c|}{$V_{h}^{*}$}\\ \hline 
         n&  $L^{2}$error &  order &  $L^{2}$error &  order &  $L^{2}$error &  order \\ \hline 
         10& 2.544e-02& -  
&9.413e-05& -   
&9.569e-05& -   
\\ \hline 
         20& 6.404e-03&1.990
&3.421e-05&1.460&3.419e-05&1.485 
\\ \hline 
         40& 1.604e-03&1.998
&9.254e-06&1.886&9.194e-06&1.895 
\\ \hline 
         80& 4.011e-04&1.999
&2.363e-06&1.970&2.345e-06&1.971 
\\ \hline 
         160& 1.003e-04&2.000
&5.939e-07&1.992&5.892e-07&1.993 
\\ \hline
         320& 2.507e-05&2.000&1.487e-07&1.998 &1.475e-07&1.998 \\ \hline
    \end{tabular}
    \caption{$H^1$error and order using $P^{2}$ Lagrange element, Poisson equation in 1d.}
    \label{tab:1.4}
\end{table}

\begin{table}[H]
    \centering
    \setlength{\tabcolsep}{8pt}
    \begin{tabular}{|c|cc|cc|cc|} \hline 
         &  \multicolumn{2}{|c|}{$V_{h}$}&  \multicolumn{2}{|c|}{$V_{h}^{+}$}&  \multicolumn{2}{|c|}{$V_{h}^{*}$}\\ \hline 
         n&  $L^{2}$error &  order &  $L^{2}$error &  order &  $L^{2}$error &  order \\ \hline 
         10& 1.527e-05& - 
&3.657e-07& - 
&3.596e-07& - 
\\ \hline 
         20& 9.554e-07&3.998
&2.614e-08&3.806&2.500e-08&3.846
\\ \hline 
         40& 5.973e-08&4.000
&1.755e-09&3.896&1.693e-09&3.885
\\ \hline 
         80& 3.733e-09&4.000
&1.119e-10&3.971&1.084e-10&3.964
\\ \hline 
         160& 2.335e-10&3.999
&7.031e-12&3.993&6.833e-12&3.988
\\ \hline
         320& 3.909e-11&2.579&4.400e-13&3.998&3.220e-12&1.085\\ \hline
    \end{tabular}
    \caption{$L^{2}$error and order using $P^{3}$ Lagrange element, Poisson equation in 1d.}
    \label{tab:1.5}
\end{table}

\begin{table}[H]
    \centering
    \setlength{\tabcolsep}{8pt}
    \begin{tabular}{|c|cc|cc|cc|} \hline 
         &  \multicolumn{2}{|c|}{$V_{h}$}&  \multicolumn{2}{|c|}{$V_{h}^{+}$}&  \multicolumn{2}{|c|}{$V_{h}^{*}$}\\ \hline 
         n&  $L^{2}$error &  order &  $L^{2}$error &  order &  $L^{2}$error &  order \\ \hline 
         10& 1.448e-03& - 
&3.482e-05& -  
&3.420e-05& -  
\\ \hline 
         20& 1.813e-04&2.998
&4.976e-06&2.807&4.761e-06&2.845 
\\ \hline 
         40& 2.266e-05&3.000
&6.668e-07&2.900&6.431e-07&2.888 
\\ \hline 
         80& 2.833e-06&3.000
&8.497e-08&2.972&8.233e-08&2.966 
\\ \hline 
         160& 3.542e-07&3.000
&1.067e-08&2.993&1.036e-08&2.991 
\\ \hline
         320& 4.427e-08&3.000&1.336e-09&2.998 &1.296e-09&2.998 \\ \hline
    \end{tabular}
    \caption{$H^{1}$error and order using $P^{3}$ Lagrange element, Poisson equation in 1d.}
    \label{tab:1.6}
    \vspace{-1.0em}
\end{table}

It can be observed that when $n$ is sufficiently large, the orders of the space $V_{h}^{+}$ and $V_{h}^{*}$ coincide with that of $V_{h}$, but are slightly lower for smaller values of $n$. This indicates that the modified finite element method demonstrates more significant improvements when the degrees of freedom are limited. Furthermore, by comparing the results in these tables, it is evident that as the polynomial order increases, the enhancement provided by the method diminishes. This behavior aligns well with our theoretical findings.

In Table \ref{tab:1.5}, when we have 320 intervals, due to the excessively large condition number of the stiffness matrix, the numerical error is ultimately constrained by the limitations of machine precision. $V_{h}^{+}$ is an exception because the solution corresponding to the linear system is small.

\subsection{biharmonic equation in One Dimension}
consider the biharmonic equation with Dirichlet boundary condition on \(\Omega = [0, 1]\) and the exact solution \( u = 3x^{2}+(x+1)\sin(4x)+1 \).
In this example, we use Hermite finite element space of degree three as the classical element space \(V_{h}\), which is compared with enriched basis \(V_{h}^{+}\) and \(V_{h}^{*}\). 

Table \ref{tab:2.1}, \ref{tab:2.2} and \ref{tab:2.3} shows the error and corresponding order for the three function spaces under consideration. A notable observation is that for space $V_{h}^{*}$, the error becomes constrained by floating-point precision limitations when $n=80$, primarily due to the deteriorating condition number of the stiffness matrix. In contrast, space $V_{h}^{+}$ can circumvent this problem to a certain extent. These results collectively indicate that our proposed method maintains its enhancement capabilities when applied to higher-order partial differential equations.
  
  \begin{table}[ht]
      \centering
      \begin{tabular}{|c|cc|cc|cc|}
        \hline
         n  & $L^{2}$error & order & $H^{1}$error & order & $H^{2}$error & order\\
         \hline
        5 & 7.608e-04& - &1.320e-02& - &4.284e-01 & - \\
        \hline
        10 & 4.861e-05&3.968&1.685e-03&2.971&1.092e-01&1.972 \\
        \hline
        20 & 3.054e-06&3.992&2.116e-04&2.993&2.743e-02&1.993 \\
        \hline
        40 & 1.911e-07&3.998&2.648e-05&2.998&6.865e-03&1.998 \\
        \hline
        80 & 1.193e-08&4.001&3.311e-06&3.000&1.717e-03&2.000 \\
        \hline
      \end{tabular}
      \caption{error and order of basis $V_{h}$, biharmonic equation in 1d.}
      \label{tab:2.1}
      \vspace{-2.0em}
  \end{table}

    \begin{table}[ht]
      \centering
      \begin{tabular}{|c|cc|cc|cc|}
        \hline
         n  & $L^{2}$error & order & $H^{1}$error & order & $H^{2}$error & order\\
         \hline
        5 & 5.526e-08& - &1.119e-06& - &3.947e-05 & - \\
        \hline
        10 & 7.729e-09&2.838&2.755e-07&2.022&1.805e-05&1.129 \\
        \hline
        20 & 5.654e-10&3.773&3.956e-08&2.800&5.126e-06&1.816 \\
        \hline
        40 & 3.543e-11&3.996&4.929e-09&3.005&1.283e-06&1.999 \\
        \hline
        80 & 2.352e-12&3.913&6.526e-10&2.917&3.387e-07&1.921 \\
        \hline
      \end{tabular}
      \caption{error and order of basis $V_{h}^{+}$, biharmonic equation in 1d.}
      \label{tab:2.2}
      \vspace{-2.0em}
  \end{table}

  \begin{table}[ht]
      \centering
      \begin{tabular}{|c|cc|cc|cc|}
        \hline
         n  & $L^{2}$error & order & $H^{1}$error & order & $H^{2}$error & order\\
         \hline
        5 & 2.153e-07& - &1.553e-06& - &4.618e-05 & - \\
        \hline
        10 & 1.570e-08&3.777&2.965e-07&2.389&1.939e-05&1.252 \\
        \hline
        20 & 1.066e-09&3.881&4.203e-08&2.819&5.426e-06&1.838 \\
        \hline
        40 & 7.968e-11&3.741&5.229e-09&3.007&1.359e-06&1.998 \\
        \hline
        80 & 4.990e-11&0.675&7.142e-10&2.872&3.576e-07&1.926 \\
        \hline
      \end{tabular}
      \caption{error and order of basis $V_{h}^{*}$, biharmonic equation in 1d.}
      \label{tab:2.3}
      \vspace{-2.0em}
  \end{table}

\subsection{Linear Elliptic Equation in Two Dimension}
  In our numerical investigation of equation (2.1), we consider three distinct cases for the coefficient function: $a(x)=1$, with $c(x)=0,\ -5$ and $-2\pi^{2}+0.01$ to see the results. For the cases where $c(x) = 0$ and $-5$, we employ $J_{r}(\theta)$ as the loss to train our PINN. Subsequently, we implement the finite element method using a structured mesh, as illustrated in figure \ref{fig:1}. The numerical errors, categorized by function space and the characteristic length of right triangles in the triangulation, are systematically presented in Tables \ref{tab:3.1} - \ref{tab:3.4}.
  \begin{figure}[H]
    \centering
    \begin{minipage}[t]{0.48\textwidth}
        \centering
        \includegraphics[width=\linewidth]{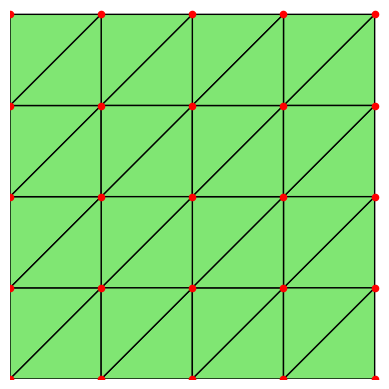}
    \end{minipage}
    \hfill 
    \begin{minipage}[t]{0.48\textwidth}
        \centering
        \includegraphics[width=\linewidth]{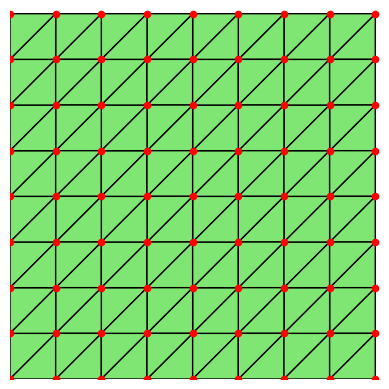}
    \end{minipage}
    \caption{left: mesh when $h=2^{-2}$, with 32 triangles and 25 nodes, right: mesh when $h=2^{-3}$, with 128 triangles and 81 nodes.}
    \label{fig:1}
\end{figure}
\begin{table}[ht]
      \centering
      \begin{tabular}{|c|c|c|c|c|c|c|c|c|}
         \hline
         basis & \multicolumn{4}{|c|}{$V_{h}$}& \multicolumn{4}{|c|}{$V_{h}^{*}$}\\
         \hline
         h   &$L^{2}$error &  order &  $H^{1}$error &  order &  $L^{2}$error &  order &  $H^{1}$error & order \\
        \hline
        $2^{-2}$ &7.394e-02&-&7.951e-01&-&1.291e-05&-&1.867e-04&-\\
        \hline
        $2^{-3}$ &1.958e-02&1.917&4.081e-01&0.962&4.553e-06&1.504&1.149e-04&0.700\\
        \hline
        $2^{-4}$ &4.969e-03&1.979&2.054e-01&0.990&1.393e-06&1.708&6.601e-05&0.800\\
        \hline
        $2^{-5}$ &1.247e-03&1.995&1.029e-01&0.998&3.693e-07&1.915&3.439e-05&0.941\\
        \hline
        $2^{-6}$ &3.120e-04&1.999&5.146e-02&0.999&9.375e-08&1.978&1.738e-05&0.985\\
        \hline
      \end{tabular}
      \caption{$L^{2}$ and $H^{1}$error and corresponding order using $P^{1}$ Lagrange element, $c(x)=0$.}
      \label{tab:3.1}
      \vspace{-2.0em}
  \end{table}

  \begin{table}[ht]
      \centering
      \begin{tabular}{|c|c|c|c|c|c|c|c|c|}
         \hline
         basis & \multicolumn{4}{|c|}{$V_{h}$}& \multicolumn{4}{|c|}{$V_{h}^{*}$}\\
         \hline
         h   &$L^{2}$error &  order &  $H^{1}$error &  order &  $L^{2}$error &  order &  $H^{1}$error & order \\
        \hline
        $2^{-2}$ &3.589e-03&-&1.134e-01&-&2.997e-06&-&8.462e-05&-\\
        \hline
        $2^{-3}$ &4.431e-04&3.018&2.895e-02&1.970&7.146e-07&2.068&3.766e-05&1.168\\
        \hline
        $2^{-4}$ &5.523e-05&3.004&7.279e-03&1.992&1.053e-07&2.762&1.116e-05&1.755\\
        \hline
        $2^{-5}$ &6.899e-06&3.001&1.822e-03&1.998&1.375e-08&2.937&2.923e-06&1.933\\
        \hline
        $2^{-6}$ &8.623e-07&3.000&4.557e-04&1.999&1.738e-09&2.984&7.398e-07&1.982 \\
        \hline
      \end{tabular}
      \caption{$L^{2}$ and $H^{1}$error and corresponding order using $P^{2}$ Lagrange element, $c(x)=0$.}
      \label{tab:3.2}
      \vspace{-2.0em}
  \end{table}

  \begin{table}[ht]
      \centering
      \begin{tabular}{|c|c|c|c|c|c|c|c|c|}
         \hline
         basis & \multicolumn{4}{|c|}{$V_{h}$}& \multicolumn{4}{|c|}{$V_{h}^{*}$}\\
         \hline
         h   &$L^{2}$error &  order &  $H^{1}$error &  order &  $L^{2}$error &  order &  $H^{1}$error & order \\
        \hline
        $2^{-2}$ &8.944e-02&-&7.995e-01&-&2.478e-05&-&3.120e-04&-\\
        \hline
        $2^{-3}$ &2.457e-02&1.864&4.088e-01&0.968&1.074e-05&1.207&2.257e-04&0.468\\
        \hline
        $2^{-4}$ &6.300e-03&1.964&2.055e-01&0.992&3.273e-06&1.714&1.312e-04&0.783\\
        \hline
        $2^{-5}$ &1.585e-03&1.991&1.029e-01&0.998&8.643e-07&1.921&6.838e-05&0.940\\
        \hline
        $2^{-6}$ &3.969e-04&1.998&5.146e-02&1.000&2.192e-07&1.980&3.456e-05&0.984 \\
        \hline
      \end{tabular}
      \caption{$L^{2}$ and $H^{1}$error and corresponding order using $P^{1}$ Lagrange element, $c(x)=-5$.}
      \label{tab:3.3}
      \vspace{-2.0em}
  \end{table}

  \begin{table}[ht]
      \centering
      \begin{tabular}{|c|c|c|c|c|c|c|c|c|}
         \hline
         basis & \multicolumn{4}{|c|}{$V_{h}$}& \multicolumn{4}{|c|}{$V_{h}^{*}$}\\
         \hline
         h   &$L^{2}$error &  order &  $H^{1}$error &  order &  $L^{2}$error &  order &  $H^{1}$error & order \\
        \hline
        $2^{-2}$ &3.775e-03&-&1.134e-01&-&7.053e-06&-&1.870e-04&-\\
        \hline
        $2^{-3}$ &4.499e-04&3.069&2.895e-02&1.970&1.463e-06&2.269&7.568e-05&1.305\\
        \hline
        $2^{-4}$ &5.545e-05&3.020&7.279e-03&1.992&2.096e-07&2.803&2.197e-05&1.784\\
        \hline
        $2^{-5}$ &6.906e-06&3.005&1.822e-03&1.998&2.715e-08&2.948&5.728e-06&1.939\\
        \hline
        $2^{-6}$ &8.625e-07&3.001&4.557e-04&1.999&3.426e-09&2.987&1.448e-06&1.984 \\
        \hline
      \end{tabular}
      \caption{$L^{2}$ and $H^{1}$error and corresponding order using $P^{2}$ Lagrange element, $c(x)=-5$.}
      \label{tab:3.4}
      \vspace{-2.0em}
  \end{table}
  The case where $c(x)=-2\pi^{2}$+0.01 presents significant computational challenges for PINN training. This difficulty arises because $2\pi^{2}$ corresponds to the first eigenvalue of the Laplace operator, rendering the training process with $J_{r}(\theta)$ as the loss function non-convergent. To address this issue, we modify the optimization framework by replacing the loss function with $J_{R}(\theta)$.
  
  The numerical errors in the L2 norm and the H1 seminorm for the different mesh and function space are shown in Table \ref{tab:3.5} and Table \ref{tab:3.6}. The results demonstrate that the accuracy of the finite element method with linear basis functions is strongly dependent on mesh refinement. When employing quadratic basis functions, the method exhibits stable and reliable performance..

  Figure \ref{fig:2d} illustrates the distribution of point-wise absolute errors computed using the two methods with an element size of $h=2^{-5}$ with quadratic basis functions. The error distribution analysis reveals a strong correlation between the PINN residuals and the multiplicative solution errors, with error concentration occurring in regions characterized by larger residuals(i.e., large error of second order derivatives). This observation suggests that adaptive mesh refinement techniques could potentially enhance the performance of our method.
  \begin{table}[ht]
      \centering
      \begin{tabular}{|c|c|c|c|c|c|c|c|c|}
         \hline
         basis & \multicolumn{4}{|c|}{$V_{h}$}& \multicolumn{4}{|c|}{$V_{h}^{*}$}\\
         \hline
         h   &$L^{2}$error &  order &  $H^{1}$error &  order &  $L^{2}$error &  order &  $H^{1}$error & order \\
        \hline
        $2^{-3}$ &4.492e-01& - &1.999e+00& - &6.153e-06& - &7.643e-05& - \\
        \hline
        $2^{-4}$ &4.311e-01&0.059&1.916e+00&0.061&8.718e-06&-0.503&6.221e-05&0.297\\
        \hline
        $2^{-5}$ &3.746e-01&0.203&1.665e+00&0.203&8.383e-06&0.057&4.574e-05&0.444\\
        \hline
        $2^{-6}$ &2.462e-01&0.606&1.094e+00&0.605&5.539e-06&0.598&2.811e-05&0.702\\
        \hline
        $2^{-7}$ &1.039e-01&1.245&4.619e-01&1.244&2.293e-06&1.272&1.227e-05&1.196 \\
        \hline
        $2^{-8}$ &3.135e-02&1.728&1.398e-01&1.724&6.844e-07&1.745&4.582e-06&1.421 \\
        \hline
      \end{tabular}
      \caption{$L^{2}$ and $H^{1}$error and corresponding order using $P^{1}$ Lagrange element, $c(x)=-2\pi^{2}$+0.01.}
      \label{tab:3.5}
      \vspace{-2.0em}
  \end{table}
  \begin{table}[ht]
      \centering
      \begin{tabular}{|c|c|c|c|c|c|c|c|c|}
         \hline
         basis & \multicolumn{4}{|c|}{$V_{h}$}& \multicolumn{4}{|c|}{$V_{h}^{*}$}\\
         \hline
         h   &$L^{2}$error &  order &  $H^{1}$error &  order &  $L^{2}$error &  order &  $H^{1}$error & order \\
        \hline
        $2^{-2}$ &3.589e-03& - &1.134e-01& - &2.665e-06& - &7.216e-05& - \\
        \hline
        $2^{-3}$ &4.431e-04&3.018&2.895e-02&1.970&7.166e-07&1.895&3.745e-05&0.946 \\
        \hline
        $2^{-4}$ &5.523e-05&3.004&7.279e-03&1.992&1.118e-07&2.680&1.197e-05&1.646\\
        \hline
        $2^{-5}$ &6.899e-06&3.001&1.822e-03&1.998&1.459e-08&2.938&3.200e-06&1.903\\
        \hline
        $2^{-6}$ &8.623e-07&3.000&4.557e-04&1.999&1.842e-09&2.986&8.146e-07&1.974\\
        \hline
        $2^{-7}$ &1.078e-07&3.000&1.139e-04&2.000&2.308e-10&2.996&2.046e-07&1.993 \\
        \hline
        $2^{-8}$ &1.347e-08&3.000&2.849e-05&2.000&2.889e-11&2.998&5.121e-08&1.998 \\
        \hline
      \end{tabular}
      \caption{$L^{2}$ and $H^{1}$error and corresponding order using $P^{2}$ Lagrange element, $c(x)=-2\pi^{2}$+0.01.}
      \label{tab:3.6}
      \vspace{-2.0em}
  \end{table}

\begin{figure}
    \centering
    \includegraphics[width=\linewidth]{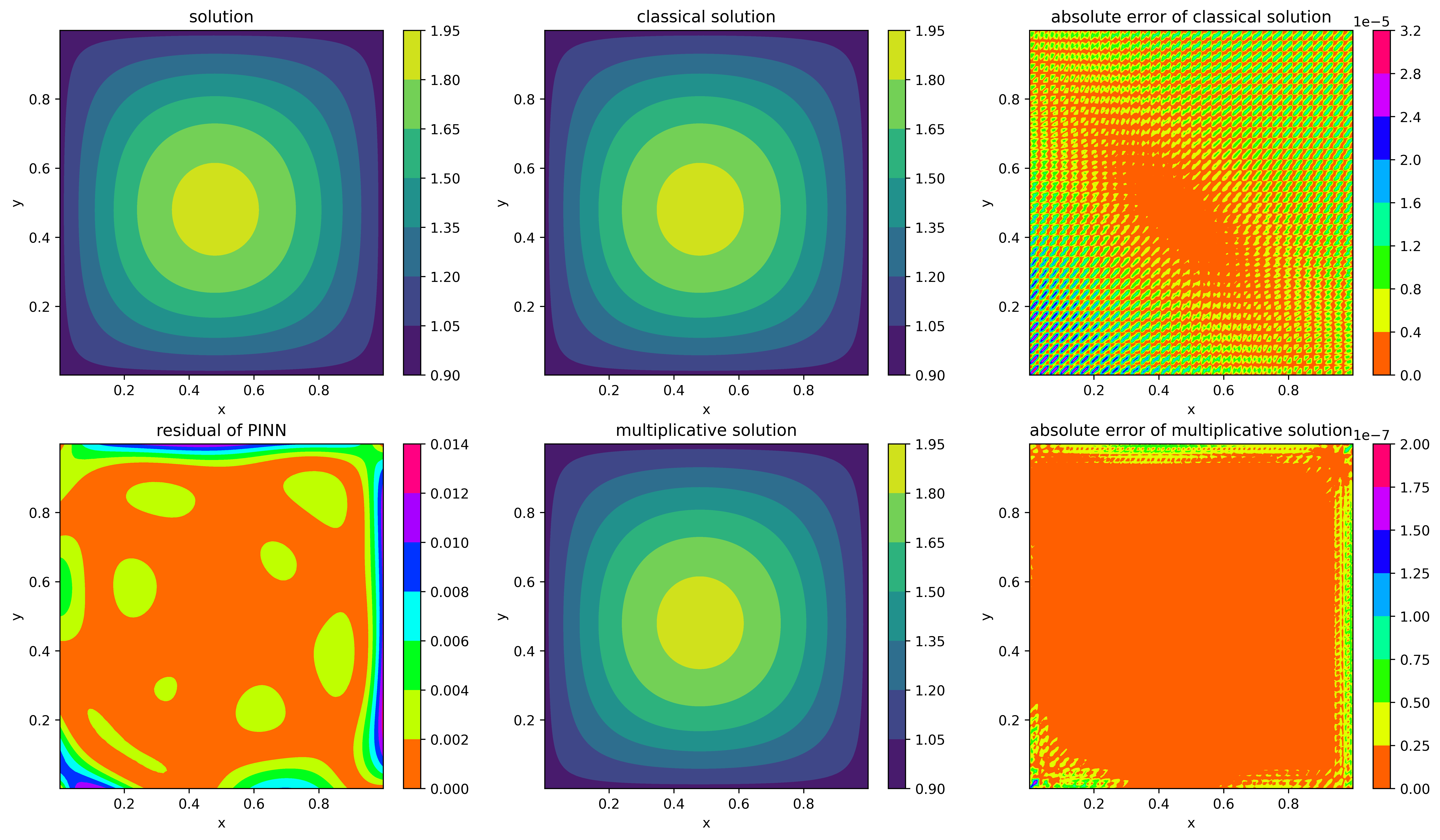}
    \caption{Comparison between solutions obtained using classical FEM and multiplicative space for cases $c(x)=-2\pi^{2}$+0.01 with $P^{2}$ Lagrange element and mesh size $h=2^{-5}$.}
    \label{fig:2d}
\end{figure}

\subsection{Poisson equation in Three Dimension}
  For our three-dimensional numerical experiments, we consider the Poisson equation with the exact solution $u(x,y,z)=e^{x+y+z}\sin(\pi x)\sin(\pi y)\sin(\pi z)+1$:
  \begin{equation}\label{5.1}
      \begin{cases}
          -\Delta u = f\quad \text{in}\ \Omega=[0,1]^{3}
          \\
          u = 1\quad \text{on}\ \partial\Omega
      \end{cases}
  \end{equation}
Table \ref{tab:4.1} lists the $L_{2}$ and $H_{1}$ errors, as well as the corresponding convergence orders. Due to memory constraints, we limit our finest mesh resolution to a tetrahedron size of $2^{-5}$, corresponding to a computational grid comprising 35937 nodes and 196608 cells.
\begin{table}[ht]
      \centering
      \begin{tabular}{|c|c|c|c|c|c|c|c|c|}
         \hline
         basis & \multicolumn{4}{|c|}{$V_{h}$}& \multicolumn{4}{|c|}{$V_{h}^{*}$}\\
         \hline
         h   &$L^{2}$error &  order &  $H^{1}$error &  order &  $L^{2}$error &  order &  $H^{1}$error & order \\
        \hline
        $2^{-1}$ &1.190& - &8.123& - &3.440e-04& - &3.623e-03& - \\
        \hline
        $2^{-2}$ &4.640e-01&1.359&5.038&0.689&1.193e-04&1.528&3.014e-03&0.265\\
        \hline
        $2^{-3}$ &1.335e-01&1.798&2.687&0.907&9.495e-05&0.330&3.069e-03&-0.026\\
        \hline
        $2^{-4}$ &3.466e-02&1.945&1.367&0.976&5.086e-05&0.900&2.257e-03&0.444\\
        \hline
        $2^{-5}$ &8.749e-03&1.986&6.863e-01&0.994&1.754e-05&1.536&1.311e-03&0.783 \\
        \hline
      \end{tabular}
      \caption{error and order using $P^{1}$ element, Poisson equation in 3d.}
      \label{tab:4.1}
  \end{table}

\section{Conclusion}
In this work, we proposed a new method combining the finite element method and physics-informed neural networks. The global approximation capability of PINNs serves to enhance the overall approximation power of the hybrid scheme. We have established a priori error estimates for both the H1 semi-norm and the L2 norm, demonstrating that the proposed method maintains the same asymptotic convergence order as classical FEM. The degree of improvement is directly correlated with the PINN's ability to accurately approximate the derivatives of solution. However, due to the inherent challenges in approximating high-dimensional functions and their higher-order derivatives, the enhancement is particularly pronounced for low-dimensional problems utilizing linear basis functions. Numerical results across various dimensions and basis functions consistently validate our theoretical findings. Furthermore, the implementation of specialized boundary operators and adaptive loss functions has significantly improved the robustness and accuracy of the PINN component.

Several promising directions for future research emerge from this work. First, the derivative information provided by PINN predictions could be leveraged to construct optimized meshes prior to FEM implementation. Second, given that the error in the modified FEM solution tends to concentrate near domain boundaries, the development of an adaptive mesh refinement strategy based on PINN predictions, rather than traditional a posteriori error estimates, warrants investigation. Third, while our current implementation focuses on linear equations, extending this framework to nonlinear problems and linear systems - such as planar elasticity or Stokes problems - presents an important research direction. Finally, the design of specialized network architectures that ensure $V_{h}^{*}$ constitutes a proper subspace of $H(\textrm{div})$ or $H(\textrm{curl})$ spaces remains an open challenge with significant potential implications.

\section{Acknowledgements}
This work was partially supported by NSFC grant 12425113.

\bibliographystyle{plain}

\newpage
\appendix
\section{improving the Training of PINN}
\subsection{Boundary Operator}
Traditional PINNs typically employ a penalty term (2.4) to enforce boundary conditions. However, this approach can lead to conflicting gradient directions between the PDE loss and boundary loss components during training, potentially hindering convergence to global optima. To address this issue, we implement a boundary operator that inherently satisfies the boundary conditions, thereby restricting the optimization process to the minimization of the PDE residual (2.3) alone. 

For equation (2.1), we define the boundary operator as
\begin{equation}\label{A1}
    \mathcal{B} (x,u_{\theta})=D(x)u_{\theta}(x)+g(x)=\overline{u}_{\theta}(x) 
\end{equation}
where $D(x)$ is a distance function associated with the domain $\Omega$. Specifically, $D(x)$ vanishes on the boundary $\partial \Omega$ while maintaining positive values within the interior of $\Omega$. This construction ensures that $\overline{u}_{\theta}$ exactly satisfies the prescribed boundary conditions, consequently yielding $J_{b}(\theta)=0$. The extension of this approach to handle Robin boundary conditions or complex geometric configurations has been effectively demonstrated in prior works\cite{berg201828,sukumar2022114333}. 

By implementing the boundary operator, the loss function simplifies to the following form:
\begin{equation}\label{A2}
    J(\theta)=J_{r}(\theta)= \int_{\Omega }( -\nabla \cdot (a(x)\nabla \overline{u}_{\theta}(x))+c(x)\overline{u}_{\theta}(x)-f(x)) ^{2}\mathrm{d}x.
\end{equation}

For case 5.1, loss function with penalty term becomes:
\begin{equation}
\begin{split}
    J(\theta) &= J_{r}(\theta) + 1000 J_{b}(\theta)
    \\
    &=\int_{\Omega }(-u^{''}_{\theta}(x)-f(x)) ^{2}\mathrm{d}x + 1000[(u_{\theta}(0)-u(0))^{2}+(u_{\theta}(1)-u(1))^{2}].
\end{split}
\end{equation}
The following figures demonstrate the significant improvement achieved through the implementation of the boundary operator for exact boundary condition enforcement. To ensure statistical reliability, we conducted five independent replications for each experimental configuration. The presented results correspond to the most frequently observed outcomes. While rare instances of comparable performance can be achieved using the loss function (A.3), as exemplified in \ref{fig:bdo}, the boundary operator approach consistently delivers superior results..
\begin{figure}[htbp]
    \centering
    \includegraphics[width=1.0\textwidth]{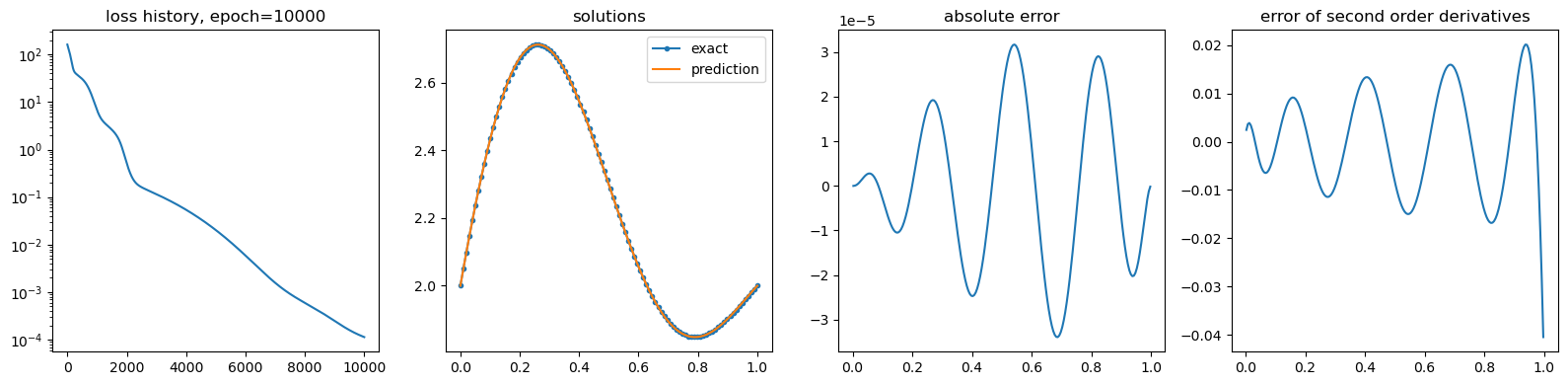}
    \caption{training with boundary operator. From left to right: training loss (A.2) with 10000 epochs using Adam; solution; error of $u-u_{\theta}$, error of second order derivative of $u-u_{\theta}$}
    \label{fig:bdo}
\end{figure}
\begin{figure}[ht]
    \centering
    \includegraphics[width=1.0\textwidth]{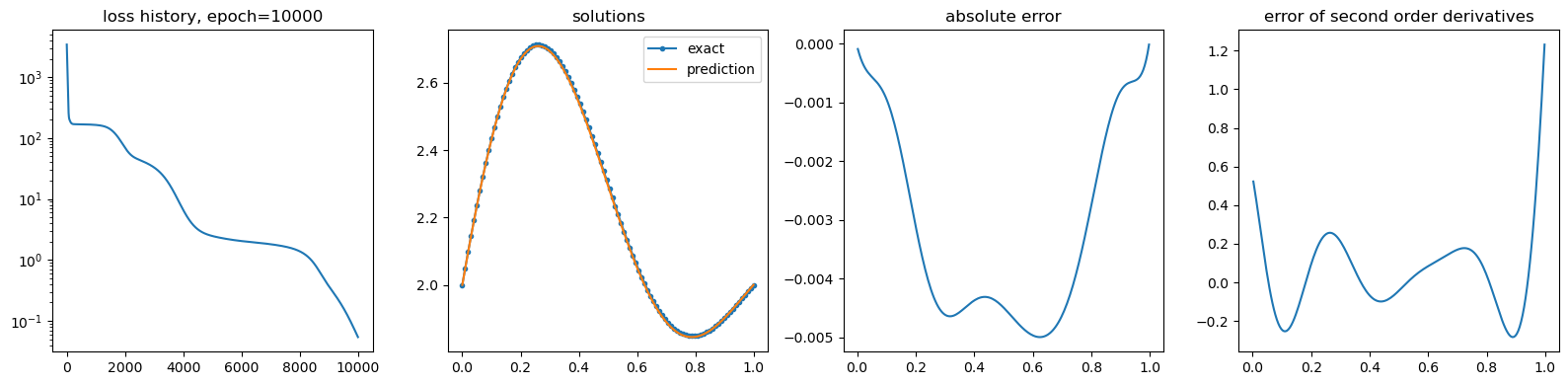}
    \caption{training with penalty term. From left to right: training loss (A.3) with 10000 epochs using Adam; solution; $u-u_{\theta}$, second order derivative of $u-u_{\theta}$.}
    \label{fig:pen}
\end{figure}

A comparative analysis of Figure \ref{fig:bdo} and Figure \ref{fig:pen} reveals that the boundary operator significantly enhances PINN training efficiency by eliminating the need for penalty terms. This improvement is particularly notable as it circumvents the common issue of slow convergence associated with local optima entrapment in traditional penalty-based approaches. The advantage becomes increasingly pronounced in high-dimensional scenarios, where the penalty method's limitations are more severe. 

Although the boundary operator is employed throughout our numerical experiments, we maintain the notation $u_{\theta}$ rather than $\overline{u}_{\theta}$ for the PINN output to preserve notational consistency and simplicity in our presentation.

\subsection{Loss Function}
Directly incorporating PDE constraints into the training process often results in prolonged training times, particularly in high-dimensional scenarios with complex solutions or derivatives. To enhance both accuracy and computational efficiency, we implement a two-phase training strategy:

1.Initial Phase: Train the PINN using the loss function $J_{r}(\theta)$. This phase continues until the loss reduction rate diminishes, yielding an intermediate solution with satisfactory $L_{2}$ error and moderate $H_{2}$ error.

2.Refinement Phase: Continue training with the loss function $J_{R}(\theta)$ until convergence, obtaining the final optimized solution.

The $J_{r}(\theta)$, loss function requires minimizing residuals across all collocation points, including second-order derivatives of the neural network. However, this approach often encounters challenges due to conflicting gradient directions at different points, making global optimization difficult. By employing (2.7) as the loss function, we can efficiently obtain a trained PINN with good $L_{2}$ and $H_{1}$ error metrics, though higher-order derivative accuracy may remain suboptimal.

The second phase, utilizing $J_{R}(\theta)$, which incorporates second-order derivatives, significantly improves the $H_{2}$ error. While this distinction is less pronounced in one-dimensional cases, we demonstrate its effectiveness through Example 5.4, comparing results from different loss functions.

To facilitate better visualization of loss function evolution, we employ logarithmic coordinates. However,  since $J_{R}(\theta)$ can assume negative values, we introduce a modified functional that remains non-negative and achieves zero only at the exact solution. This transformation is defined as follows:

\begin{equation*}
    J_{R}^{-}(\theta)=J_{R}(\theta)-J_{R}(u)
\end{equation*}
where $J_{R}(u)=\int_{\Omega }\frac{1}{2}a(x)|\nabla u(x)|^{2}+c(x)u^{2}(x)-uf(x)\mathrm{d}x$ and u is the exact solution.

\begin{figure}[htbp]
    \centering
    \includegraphics[width=1.0\textwidth]{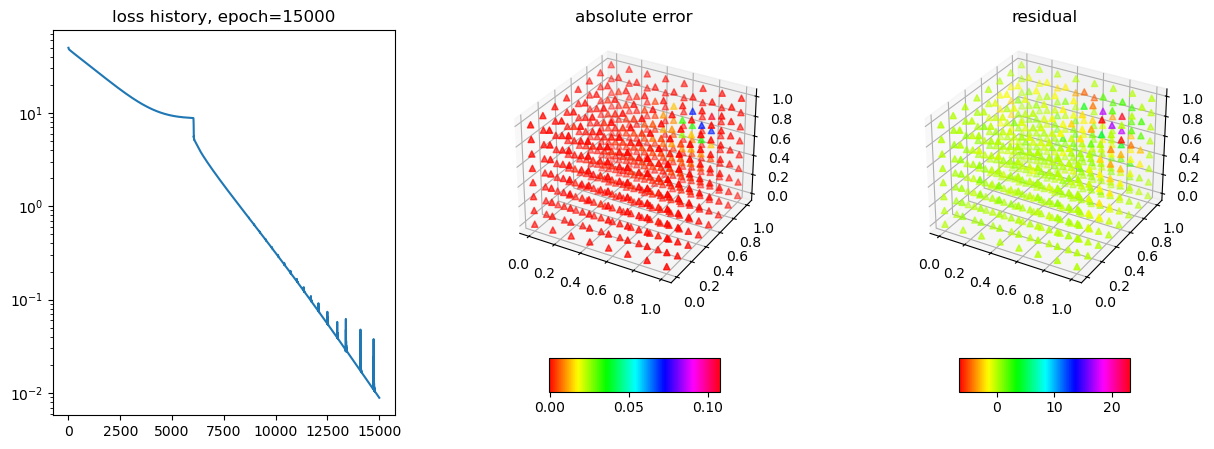}
    \caption{training with $J_{R}(\theta)$ for 15000 epochs. From left to right: training loss; $u-u_{\theta}$; residual.}
    \label{fig:Ritz}
\end{figure}

\begin{figure}[htbp]
    \centering
    \includegraphics[width=1.0\textwidth]{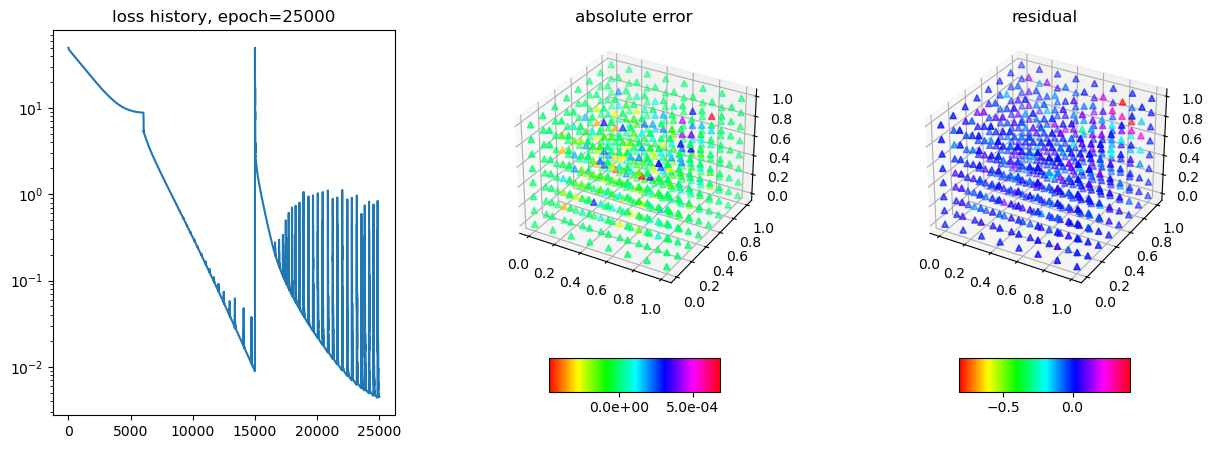}
    \caption{training with hybrid loss function: $J_{R}(\theta)$ for 15000 epochs and $J_{r}(\theta)$ for 10000 epochs. From left to right: training loss; $u-u_{\theta}$; residual.}
    \label{fig:Hybrid}
\end{figure}

\begin{figure}[htbp]
    \centering
    \includegraphics[width=1.0\textwidth]{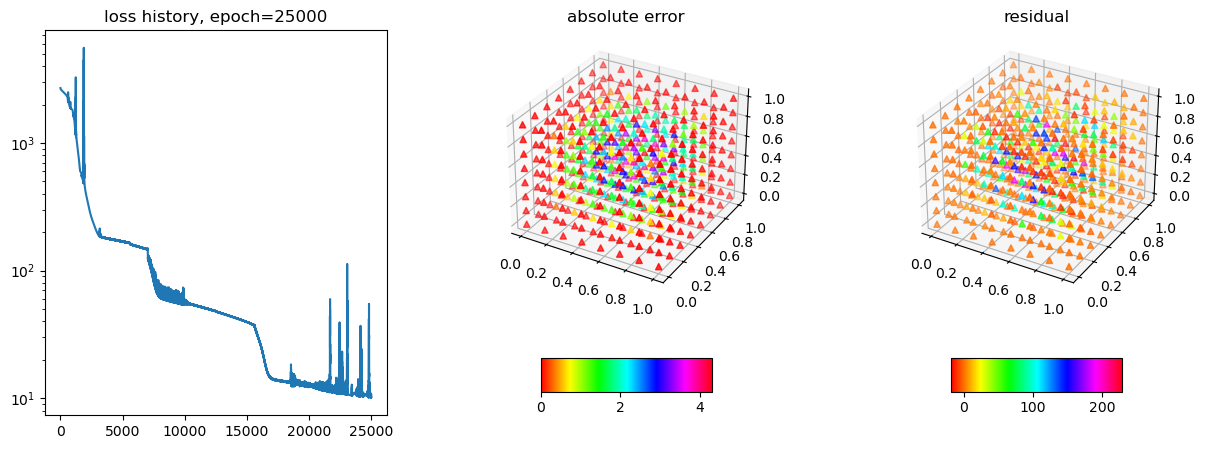}
    \caption{training with $J_{r}(\theta)$ for 25000 epochs. From left to right: training loss; $u-u_{\theta}$; residual.}
    \label{fig:PINN}
\end{figure}

Figure \ref{fig:PINN} shows that the convergence challenges encountered when using only the $J_{r}(\theta)$ loss function. As evidenced in Figure \ref{fig:Ritz}, while the Deep Ritz method achieves satisfactory results in the $L^{2}$ norm, it fails to provide comparable accuracy in the $H^{2}$ seminorm. The hybrid approach, illustrated in Figure \ref{fig:Hybrid}, reveals significant improvements when switching the loss function at epoch 15000. Notably, the solution obtained through the Deep Ritz method exhibits substantial residual errors. However, subsequent training with the $J_{r}(\theta)$ loss function for 10000 epochs achieves a remarkable two-order-of-magnitude reduction in residuals, highlighting the effectiveness of the hybrid strategy.

\end{document}